\documentclass[a4paper,11pt]{article}
\usepackage{a4}
\usepackage{amsmath}
\usepackage{amsfonts}
\usepackage{amsthm}
\usepackage{amssymb}
\makeatletter
\newtheorem{thm}{Theorem}
\newtheorem{prop}[thm]{Proposition}
\newtheorem{cor}[thm]{Corollary}
\newtheorem{lem}[thm]{Lemma}
\newtheorem{rmk}[thm]{Remark}
\newcommand\I{\mathrm i}
\newcommand\e{\mathrm e}
\renewcommand\Re{\operatorname{Re}}
\renewcommand\Im{\operatorname{Im}}
\newcommand\sgn{\operatorname{sgn}}
\newcommand\diag{\operatorname{diag}}
\makeindex
	
\begin{document}

	\label{'ubf'}
\setcounter{page}{1} %Put here the starting page number

\markboth {\hspace*{-9mm} \centerline{\footnotesize \sc
         % Put here the left page top label
     Signed Cycle Graphs
}
                 }
                { \centerline {\footnotesize \sc
                   %put here the author's name
 A.M.\ Mathai and Thomas Zaslavsky} \hspace*{-9mm}
               }
\begin{center}
{
       {\Large \textbf { \sc  On Adjacency Matrices and Descriptors of Signed Cycle Graphs}
    % Put the title of the paper here
                               }
       }
\\

\medskip

A.M.\ Mathai \footnote{Centre for Mathematical Sciences, Pala Campus, 
 Arunapuram P.O, Palai, Kerala 686 574, India, 
 and McGill University, Canada.
 cmspala@gmail.com; www.cmskerala.org; mathai@math.mcgill.ca}
 
 and
 
Thomas Zaslavsky \footnote{Department of Mathematical Sciences, Binghamton University, 
Binghamton, NY 13905-6000, U.S.A.
zaslav@math.binghamton.edu}
\end{center}

\begin{abstract}
This paper deals with adjacency matrices of signed cycle graphs and chemical descriptors based on them.  The eigenvalues and eigenvectors of the matrices are calculated and their efficacy in classifying different signed cycles is determined.  The efficacy of some numerical indices is also examined.
\end{abstract}
\large

\section{Introduction}
When graph theory is used to classify items such as chemicals it is usually done by identifying vertices and edges of a graph with various descriptors of the chemical under consideration and then devising an index based on the corresponding molecular graph under consideration. Usually a single number index is proposed. Here we will examine indices based on adjacency matrices. An index in current use is the Wiener index $W$ \cite{Refwiener47}, $\frac{1}{2}$ of the sum of all the elements in the distance matrix $D.$ Other descriptors proposed in \cite{RefMathai07} are the norms of $D.$ For example, 
\begin{equation}
\begin{aligned}
N_1&=\Vert D\Vert_1=\max_i\sum_j|d_{ij}|,\\
N_2&=\Vert D\Vert_2=\max_j\sum_i|d_{ij}|,\\
N_3&=\sqrt{\lambda},
\label{N}
\end{aligned}
\end{equation}
where $\lambda$ is the largest eigenvalue of $DD',$ $D'$ being the transpose of $D.$

Graph theory techniques are used in proteomics maps to identify or classify protein spots \cite{RefBMGH03, RefBG97}. 
In the present article we look into the problem of unique determination of the underlying pattern in a signed cycle graph (that is, a simple cycle with signed edges) through its adjacency matrix.

\section{Adjacency matrices} \label{sec:1}
A \emph{signed graph} $(G,\boldsymbol\sigma)$ is a graph $G$ with a sign function $\boldsymbol\sigma : E(G) \to \{+1,-1\}$.  
We will examine properties of some `patterned matrices' first and then these matrices will be associated with adjacency matrices of signed cycle graphs. This procedure will yield a number of results on eigenvalues and eigenvectors of adjacency matrices of  signed cycles.

Consider the following $n\times n$ circulant matrix:
\begin{equation}
B=\begin{bmatrix}
0&1&0&\dots&0&0&0\\
0&0&1&\ldots&0&0&0\\
\vdots&\vdots&\vdots&\dots&\vdots&\vdots&\vdots\\
0&0&0&\dots&0&0&1\cr 1&0&0&\dots&0&0&0\end{bmatrix}.
\label{B}
\end{equation}
Note that $B$ is real and an orthonormal matrix. $BB'=I \Rightarrow B'=B^{-1}$, where the prime denotes the transpose. Then some interesting properties follow. First, $B' = B^{n-1} = B^{-1}$. If $\nu$ is an eigenvalue of $B$ then $\nu^{-1}$ is an eigenvalue of $B^{-1}$ sharing the same eigenvectors. Hence, if we consider $A=B+B'=B+B^{-1}$ then an eigenvalue of $A$ is of the form $\lambda = \nu+\nu^{-1}$ where $\nu$ is an eigenvalue of $B$. Note that $A$ is the following patterned matrix:
\begin{equation}
A=\begin{bmatrix}
0&1&0&\dots&0&0&1\\
1&0&1&\dots&0&0&0\\
\vdots&\vdots&\vdots&\dots&\vdots&\vdots&\vdots\\
0&0&0&\dots&1&0&1\\
1&0&0&\dots&0&1&0\end{bmatrix}=B+B'=B+B^{-1}.
\label{A}
\end{equation}
But $A$ is the adjacency matrix $A(C_n)$ of a simple cycle. 

Now we consider matrices obtained from $B$ by multiplying columns by $-1$. Suppose that the $i$th column of $B$ is multiplied by $\sigma_i = (-1)^{\beta_i}$, $\beta_i\in \{0, 1\}$, so that a total of $r$ columns are multiplied by $-1$, that is, $r = \beta_{n}+\cdots+\beta_1$. Call the new matrix $B_\beta$, where $\beta = (\beta_1,\beta_2, \ldots,\beta_n)$. $B_\beta$ is still an orthonormal matrix. If $\nu$ is an eigenvalue of $B_\beta$ and $X=(z_1,\ldots,z_n)'$ is the corresponding eigenvector then 
$$B_\beta X=\nu X.$$ 
Writing the eigenvector elements explicitly we have 
\begin{equation}
\begin{aligned}
z_n &= \sigma_n \nu z_{n-1} = \sigma_n \nu \sigma_{n-1} \nu z_{n-2} = \cdots \\
&= \sigma_n \sigma_{n-1} \cdots \sigma_2 \nu^{n-1} z_1 = \sigma z_n \\
& \Rightarrow \nu^n = \sigma, 
\end{aligned}
\label{eigenelements}
\end{equation}
where 
$$\sigma = \sigma_1\sigma_2\cdots\sigma_n = \boldsymbol\sigma(C_n) = (-1)^r,$$
which means that $\nu$ is an $n$th root of $+1$ if $r$ is even and an $n$th root of $-1$ if $r$ is odd. The roots are given by 
$$\e^{\I\pi\frac{2j}{n}}, \ j=0,1,\ldots,n-1,$$ 
when $r$ is even, and 
$$\e^{\I\pi\frac{2j+1}{n}}, \ j=0,1,\ldots,n-1,$$ 
when $r$ is odd. 

Let $s \in \{0,1\}$ be such that $\sigma = (-1)^s$; then the roots are
\begin{equation}
\e^{\I\pi \frac{2j+s}{n}}, \ \ j=0,1,\ldots,n-1,
\label{roots}
\end{equation}
in all cases.
Note that $-1$ is always a root when $n$ is odd or $r$ is odd ($s=1$).  

Let $A_\beta = B_\beta+B_\beta'^{-1} = B_\beta + B_\beta^{-1}$.  Let the cycle $C_n$ have edges $e_1, e_2, \ldots, e_n$, in cyclic order.  Note that $B_{(0,\ldots,0)} = B$ and $A_{(0,\ldots,0)}=A = A(C_n)$.  
The connection between matrices $B_\beta$, $A_\beta$ and signed graphs is the following fact.

\begin{prop}\label{adjacency}
The matrix $A_\beta,$ corresponding to a choice of indices $(\beta_1,\beta_2,\ldots,\beta_n)$ is the adjacency matrix $A(C_n,\boldsymbol\sigma)$ of the signed cycle $(C_n,\boldsymbol\sigma)$ where the edge $e_i$ has sign $\boldsymbol\sigma(e_i) = \sigma_i = (-1)^{\beta_i}$.
\end{prop}

Hence, we obtain the eigenvalues of $A(C_n,\boldsymbol\sigma)$ from the following theorem.  (The eigenvalues were previously obtained by Acharya \cite[page 205]{bda} and Fan \cite[Proposition 2.2]{yf}, using different methods.)

\begin{thm}\label{thm1}  The patterned matrix $A_\beta$ has the eigenvalues 
$$\mu^\beta_j = 2\cos\frac{2j+s}{n}\pi , \ j=0,1,\ldots,n-1,$$ 
which equal 
$$2\cos\frac{2 j}{n}\pi, \ j=0,1,\ldots,n-1,$$ 
if $r$ is even ($\boldsymbol\sigma(C_n)=+1$), and 
$$2\cos\frac{2j+1}{n}\pi, \ j=0,1,\ldots,n-1,$$
if $r$ is odd ($\boldsymbol\sigma(C_n)=-1$).
 \end{thm}

\begin{proof}  The proof follows from equation \eqref{roots}.
\end{proof}

\begin{rmk}\label{pairs}{\rm 
The eigenvalues of $A_\beta$ depend only on $s$ and $n$.  

Since $\mu^\beta_{n-j-s} = \mu^\beta_j$, the eigenvalues of $A_\beta$ have multiplicity $2$, with the exception of the simple eigenvalues $\mu^\beta_0=2$ when $s=0$ (that is, $\sigma=+1$), and $\mu^\beta_{\frac{n-s}{2}}=-2$ when $n-s$ is even (that is, when $\sigma=(-1)^n$; equivalently, $\sigma=+1$ and $n$ is even, or $\sigma=-1$ and $n$ is odd).}
\end{rmk}

\begin{cor}\label{cor2}  
$2$ is an eigenvalue of $A_\beta$ when $r$ is even for all $n$, and $-2$ is an eigenvalue of $A_\beta$ when $r$ and $n$ are odd.
\end{cor}

\begin{cor}\label{cor3} 
The characteristic polynomial of $A_\beta$, denoted by $P_s(\lambda)$, depends only on $s$ and is given by
\begin{align*}
P_1(\lambda)&=\prod_{j=0}^{n-1} \Big(\lambda-2\cos\frac{2j+1}{n}\pi\Big), 
\\
P_0(\lambda)&=\prod_{j=0}^{n-1} \Big(\lambda-2\cos\frac{2j}{n}\pi\Big).
\end{align*}
\end{cor}

\begin{rmk}\label{rmk-trace}{\rm Since the trace of $A$ is zero the sum of the eigenvalues is zero.}
\end{rmk}

\begin{rmk}\label{rmk-eigen}{\rm If $n$ is even and $\lambda$ is an eigenvalue of $A$, then $-\lambda$ is also an eigenvalue of $A$.}
\end{rmk}

\begin{rmk}\label{rmk-neg}{\rm If $n$ is even then $A$ and $-A$ have the same eigenvalues.  The corresponding graph-theory statement is that if $n$ is even, negating every edge of $(C_n,\boldsymbol\sigma)$ does not change the sign $\sigma$ of $C_n$.}
\end{rmk}

\begin{rmk}\label{rmk-number}{\rm There will be $\frac{1}{2}2^n=2^{n-1}$ matrices $A_\beta$ with even $r$, having eigenvalues $2\cos\frac{2j\pi}{n}$, $j=0, 1, \ldots, n-1$, and $2^{n-1}$ matrices $A_\beta$ with odd $r$, having eigenvalues $2\cos(2j+1)\frac{\pi}{n}$, $j=0, 1, \ldots, n-1$.} 
\end{rmk}

For example, for $n=3,$ the four matrices $A_\beta$ with even $r$ have eigenvalues $2, -\sqrt{3}, -\sqrt{3}$.  The four matrices $A_\beta$ with odd $r$ have eigenvalues $\sqrt{3}, \sqrt{3}, -2$.

For $n=4,$ the eight matrices $A_\beta$ with even $r$ have eigenvalues $2$, $0$, $0$, $-2$.  The eight matrices $A_\beta$ with odd $r$ have eigenvalues $\sqrt{2}$, $\sqrt{2}$, $-\sqrt{2}$, $-\sqrt{2}.$

\section{Classification by descriptors} \label{sec:5}
Any system of classification that determines the signs of the edges in a signed cycle will uniquely determine the adjacency matrix.  
One of the criteria for classification, frequently used in the literature, is the Wiener index, which is the sum of the 
elements in the matrix under consideration. In the case of the adjacency matrices discussed above, $2^n$ in number, the Wiener index will be $2n-2r = 2n, 2n-2, \ldots, -2n$. Out of these only $2n$ and $-2n$ will uniquely determine the adjacency matrix. 
Another popular criterion is the sum of the absolute values of the eigenvalues of the matrix, the so-called \emph{energy}. In the case of a signed cycle, the eigenvalues are roots of unity and hence this sum is $n$. 
Other criteria suggested by this author are the norms of the matrix. Some of the standard norms for a matrix $A=(a_{ij})$ are $N_1, N_2, N_3$ as defined in \eqref{N}. For the adjacency matrix of a signed cycle, $N_1=N_2=2$ and hence $N_1$ and $N_2$ are poor criteria for classification. 
What about $N_3$?  In our case, where $A = A_\beta$ and $A=A',$ we have $AA'=A^2.$ 
If $\lambda$ is an eigenvalue of $A$ then $A^2$ has the eigenvalue $\lambda^2$; thus, $N_3$ equals the largest absolute value of an eigenvalue of $A$.  When either $2$ or $-2$ is an eigenvalue of $A$, $N_3 = 2.$  By Corollary \ref{cor2} that is not the case only when $n$ is odd and $r$ is even, when the largest absolute value of an eigenvalue is $2 \cos \frac{\pi}{n}$.  Therefore, $N_3$ can at most distinguish whether $r$ is even or odd.
Hence such single number criterion based on the eigenvalues will be poor in identifying the corresponding graphs.  The most efficacious one is known to be the Weiner index, which determines the value of $r$ given that of $n$ but cannot determine precisely which edges have positive or negative signs.

The spectrum or characteristic polynomial of $A(C_n,\boldsymbol\sigma)$ is less effective than the Weiner index as it determines only the value of $n$ and the parity of $r$.  It cannot distinguish more than that because it is determined by the order and sign of $(C_n,\boldsymbol\sigma)$, which are $n$ and $\boldsymbol\sigma(C_n)=(-1)^r$.  This can be explained by \emph{switching}.  Switching a vertex $v$ in a signed graph means reversing the signs of all edges incident with that vertex; it has the effect of multiplying the row and column of $v$ in $A(C_n,\boldsymbol\sigma)$ by $-1;$ in other words, $A(C_n,\boldsymbol\sigma)$ changes to $DA(C_n,\boldsymbol\sigma)D$ where $D$ is a diagonal matrix with diagonal $+1$ except for $-1$ in the row of $v$.  It is known that the spectrum is not changed by switching, and it is known that any two signings of $C_n$ with the same sign $(-1)^r$, that is the same parity of $r$, can be changed into each other by switching some set of vertices \cite{tz}.

However, the eigenvectors can be used to refine the classification.  Suppose a set $W$ of vertices is switched.  Let $D$ be the diagonal matrix with $-1$ in the rows of vertices in $W$ and $+1$ elsewhere.  Switching changes the adjacency matrix $A$ to $A_W=DAD$.  The eigenvector associated to an eigenvalue $\nu$ satisfies $AX = \nu X$ and therefore
$$A_W(DX) = \nu (DX).$$ The eigenvector after switching has elements corresponding to $W$ negated; this enables the signed graphs before and after switching to be distinguished.  Hence, the eigenvectors can distinguish the different signings of a simple cycle.  This general statement requires a proof, and there will be exceptional cases where zero is an element in an eigenvector.  Thus, we state a theorem.  

Recall that the cycle has edges $e_1,e_2,\ldots,e_n$ in cyclic order with signs $\sigma_i = \boldsymbol\sigma(e_i)$.  We show that an eigenvector of $A_\beta$ determines the edge signs.  Let the eigenvalues and eigenvectors of $A_\beta$ be $\mu^\beta_j = 2 \cos \frac{2j+s}{n} \pi$ and $x^\beta_j = (x^\beta_{j1},x^\beta_{j2},\ldots,x^\beta_{jn}) = (x^\beta_{jk})_{k=1}^n$, $j = 0,1,2,\ldots, n-1$.

\begin{thm}\label{eigenvectors}
Two independent eigenvectors of $A_\beta$ associated to the eigenvalue $\mu^\beta_j = 2 \cos \frac{2j+s}{n} \pi = \mu^\beta_{n-s-j}$, $j = 0,1,2,\ldots, n-1$, are $x^\beta_j$ and $y^\beta_j$ given by
$$x^\beta_{jk} = \cos \frac{(2j+s)k}{n}\pi$$
and
$$y^\beta_{j,k} = \sin \frac{(2j+s)k}{n}\pi$$
for $j = 1,2,\ldots, \lfloor\frac{n}{2}\rfloor$, and the eigenspace associated to $\mu^\beta_j$ is the span of $x^\beta_j$ and $y^\beta_j$; with the following exceptions.

When $\sigma=+1$ ($r$ is even, $s=0$), the eigenvectors of $A_\beta$ associated to the eigenvalue $\mu^\beta_0 = 2$ are the scalar multiples of  
$$x^\beta_0=(\sigma_1,\sigma_1\sigma_2,\ldots,\sigma_1\cdots\sigma_{n-1},\sigma_1\cdots\sigma_n)',$$ 
in which $x^\beta_{0k} = \sigma_1\cdots\sigma_k.$
If $n$ is also even, $n=2m$, then the eigenvectors of $A_\beta$ associated to the eigenvalue $\mu^\beta_m = -2$ are the scalar multiples of  
$$x^\beta_m=(-\sigma_1,\sigma_1\sigma_2,\ldots,-\sigma_1\cdots\sigma_{n-1},\sigma_1\cdots\sigma_n)',$$ 
in which $x^\beta_{mk} = (-1)^k\sigma_1\cdots\sigma_k.$

When $\sigma=-1$ ($r$ is odd, $s=1$) and $n$ is odd, $n=2m+1$, the eigenvectors of $A_\beta$ associated to the eigenvalue $\mu^\beta_m = -2$ are the scalar multiples of  
$$x^\beta_m=(-\sigma_1,\sigma_1\sigma_2,\ldots,-\sigma_1\cdots\sigma_{n-1},\sigma_1\cdots\sigma_n)',$$ 
in which $x^\beta_{mk} = (-1)^k\sigma_1\cdots\sigma_k.$
\end{thm}

\begin{proof} The proof begins by establishing properties of $B_\beta$.  Let $\delta_{ij}$ denote the Kronecker delta, $\delta_{ij} = 1$ if $i=j$ and $0$ if $i \ne j$.

\begin{lem}\label{Bpowers}
The elements of $B_\beta^k$ are $b_{\beta ij}$ given by 
$$b_{\beta ij} = \sigma_{j}\sigma_{j+1}\cdots\sigma_{j+k-1} \delta_{i+k,j},$$
where the subscripts are read modulo $n$.  Thus, $B_\beta^k = B^k D$ where $D$ is the diagonal matrix 
$$\diag(\sigma_1\cdots\sigma_k, \ \sigma_2\cdots\sigma_{k+1}, \ \ldots, \ \sigma_n\sigma_1\cdots\sigma_{k-1}).$$
\end{lem}

\begin{proof}
For $k=0, 1$ this is the definition.  (The empty product of signs $\sigma_i$ is equal to $+1$.)  The proof for $k>1$ is by induction.
\end{proof}

\begin{prop}\label{Beigen}
The eigenvector of $B_\beta$ corresponding to the eigenvalue $\nu^\beta_j = \e^{\I \frac{2j+s}{n} \pi},$ $j = 0,1,2,\ldots, n-1$, is $z^\beta_j$ whose components are 
$$
z^\beta_{jk} = \sigma_1\cdots\sigma_k \e^{\I \frac{(2j+s)k}{n} \pi}.
$$
\end{prop}

\begin{proof} By Lemma \ref{Bpowers} and the definition $\sigma_1\cdots\sigma_n = \sigma$, $(B_\beta)^n = \sigma I$.  Therefore, the characteristic polynomial of $B_\beta$ is $\lambda^n-\sigma = \lambda^n - \e^{\I \pi s}$, whose roots are $\e^{\I \frac{2j+s}{n} \pi}$, $j = 0,1,2,\ldots, n-1.$

The eigenvector corresponding to $\nu = \e^{\I \frac{2j+s}{n} \pi}$ is obtained by solving 
$$B_\beta z = (\sigma_2 z_2,\ldots,\sigma_n z_n, \sigma_1 z_1)' = (\nu z_1, \ldots, \nu z_{n-1}, \nu z_n)'$$
where $z = (z_1,\ldots,z_n)'$.  The solution is obtained as in equation \eqref{eigenelements} with the arbitrary choice $z_n=1$.
\end{proof}

The proof of Theorem \ref{eigenvectors} resumes.  $A_\beta$ has the same eigenvectors as $B_\beta$, but as $A_\beta$ is symmetric it has real eigenvalues, so we find real eigenvectors.  The real eigenvectors are the real and imaginary parts of $z^\beta_j$, which gives two eigenvectors, except when the imaginary part $\Im z^\beta_j$ has all components equal to $0$.

The real part of $z^\beta_j$ has components $\cos \frac{(2j+s)k}{n}\pi$ and the imaginary part has components $\sin \frac{(2j+s)k}{n}\pi$.  This gives the general form of $x^\beta_{jk} = \Re z^\beta_{jk} = \cos \frac{(2j+s)k}{n}\pi$ and $y^\beta_{jk} = \Im z^\beta_{jk} = \sin \frac{(2j+s)k}{n}\pi$.  
The real and imaginary parts are duplicated (up to negation) for $j > \frac{n-s}{2}$, by the identity 
\begin{equation}
z^\beta_{n-j-s} = \bar{z}^\beta_j,
\label{compvector}
\end{equation}
which implies that $x^\beta_{n-j-s} =  x^\beta_{j}$ and $y^\beta_{n-j-s} = - y^\beta_{j}$.  
Equation \eqref{compvector} is proved by
\begin{align*}
z^\beta_{n-j-s,k} &= \sigma_1\cdots\sigma_k \e^{\I \frac{(2[n-j-s]+s)k}{n} \pi} 
= \sigma_1\cdots\sigma_k \e^{\I \frac{(2n-2j-s)k}{n} \pi} \\
&= \sigma_1\cdots\sigma_k \e^{\I \frac{-(2j+s)k}{n} \pi} 
= \sigma_1\cdots\sigma_k \overline{\e^{\I \frac{(2j+s)k}{n} \pi}} = \bar z^\beta_{jk}.
\end{align*}

Therefore, eigenvalue $\mu^\beta_j$ has the eigenvectors $x^\beta_j$ and $y^\beta_j$ for $0 \leq j \leq \frac{n-s}{2}$, except when one of them is zero.  

The component $x^\beta_{j1} = \sigma_1\cdots\sigma_k \cos \frac{2j+s}{n}\pi = 0$ if and only if $\frac{(2j+s)2}{n}$ is an odd integer.  Since $0 \leq (2j+s)2 \leq 2n$, that implies $\frac{(2j+s)2}{n}=1$, thus $x^\beta_{j2} = \sigma_1\cdots\sigma_k \cos \pi = \pm1$.  Therefore, $x^\beta_j \neq 0$.  

The imaginary part $y^\beta_j$ is the zero vector when $\sin\frac{(2j+s)k}{n}\pi = 0$ for all $k$, that is, $\frac{2j+s}{n}$ is an integer.  Assuming $0 \leq 2j+s \leq n$, that is the case when $s=0$ and $n$ divides $2j$, so $j=0$, or $n$ is even and $j=\frac{n}{2}$, or when $s=1$ and $n$ is odd and divides $2j+1$, so $n=2j+1$.  These are the exceptional cases stated in the theorem.
\end{proof}

If we suppose that the information available includes the length and sign of the cycle, $n$ and $\sigma$, and an eigenvector $x$ corresponding to eigenvalue $2$ if $\sigma=+1$ or $n$ is even, or an eigenvector $x$ corresponding to eigenvalue $-2$ if $\sigma=-1$ and $n$ is odd, then we can deduce the signs of all edges from Theorem \ref{eigenvectors}.  In the former case $\sigma_k = x_k/x_{k-1}$ and in the latter case $\sigma_k = -x_k/x_{k-1}$, where $x_0$ is defined to be $x_n$.

Suppose the particular information of the previous paragraph is not available.  We can recover the edge signs from an arbitrary eigenvector $x$ using $x_1$ and the ratio $x_k/x_{k-1}$, $1<k\leq n$, if neither component is zero.  The numbers for $x^\beta_j$ are
\begin{align}
x^\beta_{j1} &= \sigma_1 \cos \frac{2j+s}{n}\pi, 
\label{initial}\\
x^\beta_{jk} \big/ x^\beta_{j,k-1} &= \sigma_k \cdot \cos\frac{(2j+s)k}{n}\pi \Big/  \cos\frac{(2j+s)(k-1)}{n}\pi. 
\label{ratios}
\end{align}
Therefore, by determining the sign of the cosine or the ratio of cosines, we can determine $\sigma_k$, $k=1,\ldots,n$.

\begin{thm}\label{edgesigns}
Let $(C_n,\boldsymbol\sigma)$ be a signed cycle graph and let $\sigma = \boldsymbol\sigma(C_n)$.

{\rm(I)}  The largest eigenvalue $\mu$ of $A(C_n,\boldsymbol\sigma)$, if known, determines $\sigma = +1$ if $\mu=2$ and $-1$ if $\mu < 2$. 

{\rm(II)}  Assume known the cosine eigenvector $x^\beta_0$ of $A(C_n,\boldsymbol\sigma)$ corresponding to $\mu$.  (Thus, $n$ is known; but $\mu$ may not be known.)  
The first coordinate $x^\beta_{01}$ determines $\sigma = (-1)^s$ and therefore $s$.  

If $s=0$ or $n$ is odd, $x^\beta_0$ determines the signs $\sigma_k$, $k=1,\ldots,n$, of all edges of $(C_n,\boldsymbol\sigma)$.  

If $s=1$ and $n$ is even, $n=2m$, then $x^\beta_0$ determines the signs $\sigma_k$ of all edges of $(C_n,\boldsymbol\sigma)$ except $\sigma_{m-1}$ and $\sigma_m$.  

{\rm(III)}  Assume known the sine eigenvector $y^\beta_0$ of $A(C_n,\boldsymbol\sigma)$ corresponding to $\mu$.  (Thus, $n$ is known; but $\mu$ may not be known.)  

If $s=1$, then $y^\beta_0$  determines the signs $\sigma_k$ of all edges of $(C_n,\boldsymbol\sigma)$ except $\sigma_1$ and $\sigma_n$.  

{\rm(IV)} If both $x^\beta_0$ and $y^\beta_0$ are known, then all edge signs are determined.
\end{thm}

\begin{proof}
(I)  The eigenvalues are $\mu^\beta_j = 2\cos\frac{2j+s}{n}\pi$, $0\leq j < \frac{n-s}{2}$.  Since $0 \leq 2j+s < n$, the angle is in the interval $[0,\pi)$.  Thus, the largest eigenvalue is $\mu^\beta_0 = 2\cos\frac{s}{n}\pi$.  This eigenvalue equals $2 \Leftrightarrow s=0$.

(II)  First, consider $x^\beta_{01} = \sigma_1 \cos\frac{s}{n}\pi$.  Because $s < n/2$, $\cos\frac{s}{n}\pi$ is positive so
\begin{equation}
\sigma_1 = \sgn x^\beta_{01}.
\label{sigma01}
\end{equation}
Also, $|x^\beta_{01}| = 1 \Leftrightarrow s=0$, so $x^\beta_0$ determines $s$, which determines $\sigma = (-1)^s$.

If $\cos\frac{s(k-1)}{n}\pi, \cos\frac{sk}{n}\pi \neq 0$, then 
\begin{equation}
\sigma_k = \frac{ x^\beta_{0k} / x^\beta_{0,k-1} }{ \cos\frac{sk}{n}\pi \big/  \cos\frac{s(k-1)}{n}\pi } \ .
\label{sigma0k}
\end{equation}
Therefore, when $s=0$, the edge signs are determined by 
\begin{equation}
\sigma_k = x^\beta_{0k} / x^\beta_{0,k-1}, \ k = 2,\ldots,n .
\label{sigma0k0}
\end{equation}
When $s=1$, 
\begin{equation}
\sigma_k = \frac{ x^\beta_{0k} / x^\beta_{0,k-1} }{ \cos\frac{k}{n}\pi \big/  \cos\frac{k-1}{n}\pi } \ .
\label{sigma0k1}
\end{equation}
If both cosines have the same sign, their quotient is positive.  Since the cosine changes sign at angle $\frac\pi2$, the quotient is positive when $\frac{k}{n} < \frac12$ or $\frac{k-1}{n} > \frac12$, that is, $k<\frac{n}{2}$ or $k-1>\frac{n}{2}$.  Thus, 
\begin{equation}
\sigma_k = \sgn(x^\beta_{0k} / x^\beta_{0,k-1}), \ 2\leq k < \frac{n}{2} \text{ or } \frac{n}{2}+1 < k \leq n,
\label{sigma0k1}
\end{equation}
for any values of $s$ and $n$.

When $n$ is odd, there is one exceptional value $k=\frac{n+1}{2}$, at which the quotient of cosines is negative.  Therefore, 
\begin{equation}
\sigma_m = -\sgn(x^\beta_{0m} / x^\beta_{0,m-1}), \ \text{ when } s=1 \text{ and } n=2m-1.
\label{sigma0k1.5}
\end{equation}

When $s=1$ and $n$ is even, $n=2m$, the value $x^\beta_{0m} = 0$.  Then the quotient of cosines in equation \eqref{sigma0k} is undefined so $\sigma_m$ and $\sigma_{m+1}$ cannot be deduced from that equation.  

(III)  Assume $s=1$.  (Otherwise, no information can be obtained from $y^\beta_0$ because $y^\beta_0=0$.)  
The components of $y^\beta_0$ are $y^\beta_{0k} = \sin\frac{k}{n}\pi$, $k=1,\ldots,n$.  No component is negative.  The only zero components are $y^\beta_{01}=0$ when $k=n$.  Therefore,
\begin{equation}
\sigma_k = \sgn(y^\beta_{0k} / y^\beta_{0,k-1}), \ k = 2,\ldots,n-1.
\label{sigma0ky}
\end{equation}

(IV)  The only signs $\sigma_k$ that may not be determined by $x^\beta_0$ are $\sigma_m$ and $\sigma_{m+1}$ when $s=1$ and $n=2m$.  Since $2\leq m$ and $m+1<n$, these signs are determined by $y^\beta_0$.
\end{proof}

The question arises of how to determine the cosine and sine vectors if only the eigenspace of $\mu^\beta_0$ is known.  If the eigenspace has dimension $1$, then $\sigma=+1$ (and $s=0$) and $x^\beta_0$ is the unique eigenvector whose last component equals $+1$.  If the eigenspace has dimension $2$, so $\sigma=-1$ (and $s=1$), $x^\beta_0$ is only one of an infinite number of vectors whose last component equals $+1$.  All such vectors have the form $x^\beta_0+\alpha y^\beta_0$, $\alpha \in \mathbf R$.  It appears that there is no way to determine $x^\beta_0$ itself from the available information.  

Once the values of all $\sigma_k$, $k=1,\ldots,n$, are known, there is one more practical difficulty.  The sign sequence $\boldsymbol\sigma = (\sigma_1,\sigma_2,\ldots,\sigma_n)$ can depend on the arbitrary choice of edge labeling.  There are $n$ possible labelings if the cycle is oriented, and $2n$ if it is not.  In order to compare different labeled cycles of the same length, one has to rotate the first sign sequence through all patterns, $(\sigma_2,\ldots,\sigma_n,\sigma_1)$ etc.  If any rotation matches the second sequence, the two cycles are isomorphic in the specified orientations, but if not, they are not isomorphic in their orientations.  If orientation is not a concern, the first sequence should also be reversed and then the $n$ rotations compared to the second sequence.  If there is a match, the two cycles are isomorphic, and if not, they are nonisomorphic.  (The amount of work in this comparison is small, of order $n$.)

\section{Conclusion}\label{sec:6}
An attempt is made in this article to come up with a spectral criterion so that the underlying pattern is uniquely determined in the case of a given signed cycle. When classifying chemicals or other items by looking at some patterned matrices in general, such a criterion will help in grouping atoms in the molecule modelled by its adjacency matrix. It is also illustrated that the single number indices are very poor for the purpose of classification. The idea is to come up with less limited but still a small amount of descriptive data, such as one or two eigenvectors of the adjacency matrix of the given molecular structure modeled by a signed graph, so that the underlying pattern can be uniquely determined. This eigenvector method, based on the adjacency matrix, is successful, but it may not yet give the minimum number of specifiers. Further, an extension of this method to wider classes of signed graphs seems to hold potential promise to devise spectral methods to find new topological descriptors for a wide variety of chemical compounds.

\vspace{0.50cm}
\begin{center} \textbf{Acknowledgements} \end{center}
The first author is grateful to the Department of Science and Technology (DST), Government of India, New Delhi, for the financial assistance for this work, done under project-number SR/S4/MS:287/05 as also to the Centre for Mathematical Sciences for providing all the required facilities for the purpose.

\end{document}